\documentclass[12pt,reqno,draft]{amsart} 
\usepackage{amssymb,amscd,url}

\begin{document}

\allowdisplaybreaks

\newif\ifdraft 
\drafttrue
\newcommand{\DRAFTNUMBER}{5}

\newcommand{\DATE}{\today\ \ifdraft(Draft \DRAFTNUMBER)\fi}
\newcommand{\TITLE}{On a Dynamical Brauer--Manin Obstruction}
\newcommand{\TITLERUNNING}{On a Dynamical Brauer--Manin Obstruction}




\newtheorem{theorem}{Theorem}
\newtheorem{lemma}[theorem]{Lemma}
\newtheorem{conjecture}[theorem]{Conjecture}
\newtheorem{proposition}[theorem]{Proposition}
\newtheorem{corollary}[theorem]{Corollary}
\newtheorem*{claim}{Claim}

\theoremstyle{definition}
\newtheorem{question}{Question}
\renewcommand{\thequestion}{\Alph{question}} 
\newtheorem*{definition}{Definition}
\newtheorem{example}[theorem]{Example}

\theoremstyle{remark}
\newtheorem{remark}[theorem]{Remark}
\newtheorem*{acknowledgement}{Acknowledgements}



\newenvironment{notation}[0]{%
  \begin{list}%
    {}%
    {\setlength{\itemindent}{0pt}
     \setlength{\labelwidth}{4\parindent}
     \setlength{\labelsep}{\parindent}
     \setlength{\leftmargin}{5\parindent}
     \setlength{\itemsep}{0pt}
     }%
   }%
  {\end{list}}

\newenvironment{parts}[0]{%
  \begin{list}{}%
    {\setlength{\itemindent}{0pt}
     \setlength{\labelwidth}{1.5\parindent}
     \setlength{\labelsep}{.5\parindent}
     \setlength{\leftmargin}{2\parindent}
     \setlength{\itemsep}{0pt}
     }%
   }%
  {\end{list}}
\newcommand{\Part}[1]{\item[\upshape#1]}

\renewcommand{\a}{\alpha}
\renewcommand{\b}{\beta}
\newcommand{\g}{\gamma}
\renewcommand{\d}{\delta}
\newcommand{\e}{\epsilon}
\newcommand{\f}{\varphi}
\newcommand{\bfphi}{{\boldsymbol{\f}}}
\renewcommand{\l}{\lambda}
\renewcommand{\k}{\kappa}
\newcommand{\lhat}{\hat\lambda}
\newcommand{\m}{\mu}
\newcommand{\bfmu}{{\boldsymbol{\mu}}}
\renewcommand{\o}{\omega}
\renewcommand{\r}{\rho}
\newcommand{\rbar}{{\bar\rho}}
\newcommand{\s}{\sigma}
\newcommand{\sbar}{{\bar\sigma}}
\renewcommand{\t}{\tau}
\newcommand{\z}{\zeta}

\newcommand{\D}{\Delta}
\newcommand{\G}{\Gamma}
\newcommand{\F}{\Phi}

\newcommand{\ga}{{\mathfrak{a}}}
\newcommand{\gA}{{\mathfrak{A}}}
\newcommand{\gb}{{\mathfrak{b}}}
\newcommand{\gm}{{\mathfrak{m}}}
\newcommand{\gn}{{\mathfrak{n}}}
\newcommand{\go}{{\mathfrak{o}}}
\newcommand{\gO}{{\mathfrak{O}}}
\newcommand{\gp}{{\mathfrak{p}}}
\newcommand{\gP}{{\mathfrak{P}}}
\newcommand{\gq}{{\mathfrak{q}}}
\newcommand{\gR}{{\mathfrak{R}}}

\newcommand{\Abar}{{\bar A}}
\newcommand{\Ebar}{{\bar E}}
\newcommand{\Kbar}{{\bar K}}
\newcommand{\Pbar}{{\bar P}}
\newcommand{\Sbar}{{\bar S}}
\newcommand{\Tbar}{{\bar T}}
\newcommand{\ybar}{{\bar y}}
\newcommand{\phibar}{{\bar\f}}

\newcommand{\Acal}{{\mathcal A}}
\newcommand{\Bcal}{{\mathcal B}}
\newcommand{\Ccal}{{\mathcal C}}
\newcommand{\Dcal}{{\mathcal D}}
\newcommand{\Ecal}{{\mathcal E}}
\newcommand{\Fcal}{{\mathcal F}}
\newcommand{\Gcal}{{\mathcal G}}
\newcommand{\Hcal}{{\mathcal H}}
\newcommand{\Ical}{{\mathcal I}}
\newcommand{\Jcal}{{\mathcal J}}
\newcommand{\Kcal}{{\mathcal K}}
\newcommand{\Lcal}{{\mathcal L}}
\newcommand{\Mcal}{{\mathcal M}}
\newcommand{\Ncal}{{\mathcal N}}
\newcommand{\Ocal}{{\mathcal O}}
\newcommand{\Pcal}{{\mathcal P}}
\newcommand{\Qcal}{{\mathcal Q}}
\newcommand{\Rcal}{{\mathcal R}}
\newcommand{\Scal}{{\mathcal S}}
\newcommand{\Tcal}{{\mathcal T}}
\newcommand{\Ucal}{{\mathcal U}}
\newcommand{\Vcal}{{\mathcal V}}
\newcommand{\Wcal}{{\mathcal W}}
\newcommand{\Xcal}{{\mathcal X}}
\newcommand{\Ycal}{{\mathcal Y}}
\newcommand{\Zcal}{{\mathcal Z}}

\renewcommand{\AA}{\mathbb{A}}
\newcommand{\BB}{\mathbb{B}}
\newcommand{\CC}{\mathbb{C}}
\newcommand{\FF}{\mathbb{F}}
\newcommand{\GG}{\mathbb{G}}
\newcommand{\NN}{\mathbb{N}}
\newcommand{\PP}{\mathbb{P}}
\newcommand{\QQ}{\mathbb{Q}}
\newcommand{\RR}{\mathbb{R}}
\newcommand{\ZZ}{\mathbb{Z}}

\newcommand{\bfa}{{\mathbf a}}
\newcommand{\bfb}{{\mathbf b}}
\newcommand{\bfc}{{\mathbf c}}
\newcommand{\bfe}{{\mathbf e}}
\newcommand{\bff}{{\mathbf f}}
\newcommand{\bfg}{{\mathbf g}}
\newcommand{\bfp}{{\mathbf p}}
\newcommand{\bfr}{{\mathbf r}}
\newcommand{\bfs}{{\mathbf s}}
\newcommand{\bft}{{\mathbf t}}
\newcommand{\bfu}{{\mathbf u}}
\newcommand{\bfv}{{\mathbf v}}
\newcommand{\bfw}{{\mathbf w}}
\newcommand{\bfx}{{\mathbf x}}
\newcommand{\bfy}{{\mathbf y}}
\newcommand{\bfz}{{\mathbf z}}
\newcommand{\bfA}{{\mathbf A}}
\newcommand{\bfF}{{\mathbf F}}
\newcommand{\bfB}{{\mathbf B}}
\newcommand{\bfD}{{\mathbf D}}
\newcommand{\bfG}{{\mathbf G}}
\newcommand{\bfI}{{\mathbf I}}
\newcommand{\bfM}{{\mathbf M}}
\newcommand{\bfzero}{{\boldsymbol{0}}}

\newcommand{\Adele}{\textsf{\upshape A}}
\newcommand{\Ahat}{\hat{A}}
\newcommand{\Adot}{A(\Adele_K)_{\bullet}}
\newcommand{\Aut}{\operatorname{Aut}}
\newcommand{\Br}{\operatorname{Br}}  
\newcommand{\Closure}{\textsf{\upshape C}} 
\newcommand{\Disc}{\operatorname{Disc}}
\newcommand{\Div}{\operatorname{Div}}
\newcommand{\End}{\operatorname{End}}
\newcommand{\Fbar}{{\bar{F}}}
\newcommand{\FOD}{\textup{FOM}}
\newcommand{\FOM}{\textup{FOD}}
\newcommand{\Gal}{\operatorname{Gal}}
\newcommand{\GL}{\operatorname{GL}}
\newcommand{\Index}{\operatorname{Index}}
\newcommand{\Image}{\operatorname{Image}}
\newcommand{\liftable}{{\textup{liftable}}}
\newcommand{\hhat}{{\hat h}}
\newcommand{\Ksep}{K^{\textup{sep}}}
\newcommand{\Ker}{{\operatorname{ker}}}
\newcommand{\Lsep}{L^{\textup{sep}}}
\newcommand{\Lift}{\operatorname{Lift}}
\newcommand{\LS}[2]{{\genfrac{(}{)}{}{}{#1}{#2}}} 
\newcommand{\vlim}{\operatornamewithlimits{\text{$v$}-lim}}
\newcommand{\wlim}{\operatornamewithlimits{\text{$w$}-lim}}
\newcommand{\MOD}[1]{~(\textup{mod}~#1)}
\newcommand{\Norm}{{\operatorname{\mathsf{N}}}}
\newcommand{\notdivide}{\nmid}
\newcommand{\normalsubgroup}{\triangleleft}
\newcommand{\odd}{{\operatorname{odd}}}
\newcommand{\onto}{\twoheadrightarrow}
\newcommand{\Orbit}{\mathcal{O}}
\newcommand{\ord}{\operatorname{ord}}
\newcommand{\Per}{\operatorname{Per}}
\newcommand{\PrePer}{\operatorname{PrePer}}
\newcommand{\PGL}{\operatorname{PGL}}
\newcommand{\Pic}{\operatorname{Pic}}
\newcommand{\Prob}{\operatorname{Prob}}
\newcommand{\Qbar}{{\bar{\QQ}}}
\newcommand{\rank}{\operatorname{rank}}
\newcommand{\Resultant}{\operatorname{Res}}
\renewcommand{\setminus}{\smallsetminus}
\newcommand{\Span}{\operatorname{Span}}
\newcommand{\tors}{{\textup{tors}}}
\newcommand{\Trace}{\operatorname{Trace}}
\newcommand{\twistedtimes}{\mathbin{%
   \mbox{$\vrule height 6pt depth0pt width.5pt\hspace{-2.2pt}\times$}}}
\newcommand{\UHP}{{\mathfrak{h}}}    
\newcommand{\Vdot}{V(\Adele_K)_{\bullet}}
\newcommand{\Wreath}{\operatorname{Wreath}}
\newcommand{\<}{\langle}
\renewcommand{\>}{\rangle}

\newcommand{\longhookrightarrow}{\lhook\joinrel\longrightarrow}
\newcommand{\longonto}{\relbar\joinrel\twoheadrightarrow}

\newcommand{\Spec}{\operatorname{Spec}}
\renewcommand{\div}{{\operatorname{div}}}

\newcounter{CaseCount}
\Alph{CaseCount}
\def\Case#1{\par\vspace{1\jot}\noindent
\stepcounter{CaseCount}
\framebox{Case \Alph{CaseCount}.\enspace#1}
\par\vspace{1\jot}\noindent\ignorespaces}

\title[\TITLERUNNING]{\TITLE}
\date{\DATE}

\author[Hsia, Silverman]{Liang-Chung Hsia and Joseph H. Silverman}
\address{Department of Mathematics,
         National Central University,
         Chung-Li, 32054 Taiwan, R. O. C.}
\email{hsia@math.ncu.edu.tw}
\address{Mathematics Department, 
         Box 1917,
         Brown University, 
         Providence, RI 02912 USA}
\email{jhs@math.brown.edu}
\thanks{The first author's research supported under project 
  NSC 95-2115-M-008-002 of the National Science Council of Taiwan. \\
  The second author's research supported by NSA H98230-04-1-0064
and NSF DMS-0650017.}

\begin{abstract}
Let $\f:X\to X$ be a morphism of a variety defined over a number
field~$K$, let~$V\subset X$ be a $K$-subvariety, and
let~$\Ocal_\f(P)=\{\f^n(P):n\ge0\}$ be the orbit of a point~$P\in
X(K)$.  We describe a local-global principle for the
intersection~$V\cap\Ocal_\f(P)$.  This principle may be viewed as a
dynamical analog of the Brauer--Manin obstruction.  We show that the
rational points of~$V(K)$ are Brauer--Manin unobstructed for power
maps on~$\PP^2$ in two cases: (1)~$V$ is a translate of a torus.
(2)~$V$ is a line and~$P$ has a preperiodic coordinate.  A key tool in
the proofs is the classical Bang--Zsigmondy theorem on primitive
divisors in sequences.  We also prove analogous local-global results
for dynamical systems associated to endomoprhisms of abelian
varieties.
\end{abstract}


\maketitle

\section*{Introduction}

An important part of the field of arithmetic dynamics is the study of the
arithmetic properties of algebraic points under iteration of maps on
algebraic varieties. Many of the fundamental problems in this subject
are transpositions to a dynamical setting of classical results and
conjectures in the theory of Diophantine equations. A key tool in the
study of Diophantine equations is the application of local-global
principles, such as the Hasse principle and the Brauer--Manin
obstruction. In this paper we begin to develop a local-global theory
for arithmetic dynamics.
\par
Our starting point is a beautiful recent result of
Scharaschkin, who showed in his
thesis~\cite{scharaschkin} that the Brauer--Manin obstruction for
rational points on curves of genus at least~$2$ can be reformulated in
non-cohomological terms as a purely ad\`elic-geometric
statement. (See~\cite{volochlocalglobal} for an analysis of
Scharaschkin's criterion over function fields.)  A straightforward
translation of Scharaschkin's ideas into the dynamical setting
yields the following criterion.

Let~$K$ be a number field, let~$X/K$ be a projective variety, and let
\[
  \f:X\to X
\]
be a $K$-morphism of infinite order. Let~$\Adele_K$ denote the ring of
ad\`eles of~$K$, and for any point~$P\in X(K)$,
write~$\Closure\bigl(\Orbit_\f(P)\bigr)$ for the closure of the
orbit~$\Orbit_\f(P)$ of~$P$ in~$X(\Adele_K)$.  Let~$V$ be a subvariety
of~$X$ that contains no nontrivial $\f$-preperiodic subvarieties (as
defined in Section~\ref{section:defsnotation}).

\begin{definition}
\textup{[Dynamical Brauer--Manin Obstruction]}
\label{wanderingquestion}
With the above notation, we say that $V(K)$ is \emph{Brauer--Manin
unobstructed} (\emph{for~$\f$}) if for every point~$P\in X(K)$ we have
\[
  \Orbit_\f(P) \cap V(K) = \Closure\bigl(\Orbit_\f(P)\bigr) \cap V(\Adele_K).
\]
\end{definition}


In Section~\ref{section:transtori} (Theorem~\ref{thm:transtori}) we
show that~$V(K)$ is Brauer--Manin unobstructed for
the~$d^{\text{th}}$-power map on~$\PP^2$ and varieties~$V$ that are
translates of tori in~$\GG_m^2$.  We also give partial results in
Section~\ref{section:powerline} (Theorem~\ref{thm:Pprepercoord}) in
the case that~$V$ is an arbitrary line in~$\PP^2$.  In
Section~\ref{section:transabvar} (Theorem~\ref{thm:transav}) we show
that~$V(K)$ is Brauer--Manin unobstructed for the the
multiplication-by-$d$ map on an abelian variety when~$V$ is a
translate of a codimension~$1$ abelian subvariety of~$A$.  All of
these result rely on the existence of primitive divisors
(Bang--Zsigmondy type theorems).  Finally, in Section~\ref{section:av}
(Theorem~\ref{thm:abv}) we use results of Serre~\cite{MR0289513} and
Stoll~\cite{arxiv0606465v3} to study the dynamical Brauer--Manin
obstruction for multiplication maps and more general subvarieties of
abelian varieties.

\begin{remark}
Zhang~\cite[Remark~4.2.3]{zhang} has also studied algebraic dynamics
over the ad\`eles, although the questions that he raises have a somewhat
different flavor from those considered in this paper.
Let~$\f:\PP^N\to\PP^N$ be a morphism of degree~$d\ge2$ defined over a
number field~$K$ and let~$S$ be a finite set of places of~$K$. For
each~$v\in S$ there is a canonical probability measure~$\mu_v$
on~$\hat\PP^N(\CC_v)$ attached to~$\f$. (Here~$\hat\PP^N$ is~$\PP^N$
if~$v$ is archimedean and~$\hat\PP^N$ is Berkovich projective space
if~$v$ is nonarchimedean. In any case, the construction of~$\mu_v$ is
nontrivial.)  Let~$P_1,P_2,\ldots\in\PP^N(\Kbar)$ be a sequence of
algebraic points such that no infinite subsequence is contained in a
preperiodic subvariety of~$\PP^N$ and assume further
that~$\hhat_\f(P_i)\to0$ as~$i\to\infty$, where~$\hhat_\f$ is the
canonical height associated to~$\f$. Then Zhang conjectures that the
the set~$\{P_i\}_{i\ge1}$ is equidistributed in~$\prod_{v\in
S}\hat\PP^N(\CC_v)$ with respect to the product measure~$\prod_{v\in
S}\mu_v$.
\end{remark}

\begin{acknowledgement}
The authors would like to thank Felipe Voloch for making available a
preprint of his paper~\cite{volochlocalglobal} and Mike Rosen for his
helpful suggestions.  The second author would also like to thank Jing
Yu, Yen-Mei Chen, his coauthor, and the NCTS for their hospitality
during his visit when this work was initiated.
\end{acknowledgement}

\section{Definitions and Notation}
\label{section:defsnotation}
We set the following notation, which will remain fixed throughout
this paper.

\begin{notation}
\setlength{\itemsep}{3pt}
\item[$K$]
a number field.
\item[$M_K$]
the set of inequivalent places of $K$.
\item[$M_K^\infty$]
the set of archimedean places of $K$.
\item[$\gp_v$]
the prime ideal associated to a finite place~$v\in M_K$.
\item[$\Adele_K$]
the ring of ad\`eles of $K$.
\item[$X/K$]
a projective variety.
\item[$\f$]
a morphism $\f:X\to X$ defined over~$K$.
\item[$V$]
a subvariety of~$X$, defined over~$K$.
\item[$\Orbit_\f(P)$]
The (forward) orbit of a point $P\in X$ under iteration of~$\f$.
\item[$\Closure\bigl(\Orbit_\f(P)\bigr)$]
The ad\`elic closure of~$\Orbit_\f(P)$ in $X(\Adele_K)$.
\end{notation}

\begin{definition}
A subvariety~$W$ of~$X$ is said to be \emph{$\f$-preperiodic}
if there are integers~$n>m$ such that $\f^n(W)=\f^m(W)$. If
also~$\dim(W)\ge1$, we say that~$W$ is \emph{nontrivial}.
\end{definition}

It is clear that we have an inclusion
\[
  \Orbit_\f(P) \cap V(K) \subseteq
  \Closure\bigl(\Orbit_\f(P)\bigr) \cap V(\Adele_K),
\]
since~$\Orbit_\f(P)$ is contained in its closure and~$V(K)$ is
contained in~$V(\Adele_K)$. A point in the right-hand side is given by
local data, and the following Brauer--Manin property asks if this
local data is sufficient to characterize the set of global points.

\begin{definition}
\label{dynbrauermanin}
With notation as above, let~$V_\f^{\text{pp}}$ be the union of all
nontrivial~$\f$-preperiodic subvarieties of~$V$.  Then we say
that~$V(K)$ is \emph{Brauer--Manin unobstructed} (\emph{for~$\f$}) if
for every point~$P\in X(K)$ satisfying \text{$\Orbit_\f(P)\cap
V_\f^{\text{pp}}=\emptyset$} we have
\[
  \Orbit_\f(P) \cap V(K) = \Closure\bigl(\Orbit_\f(P)\bigr) \cap V(\Adele_K).
\]
\end{definition}

Since we are assuming that~$X$ and~$V$ are projective, they are
proper over~$K$, so their sets of ad\`elic points are simply the
products
\[
  X(\Adele_K) = \prod_{v\in M_K} X(K_v)
  \qquad\text{and}\qquad
  V(\Adele_K) = \prod_{v\in M_K} V(K_v).
\]
Thus for example, a point~$Q\in X(\Adele_K)$ has the
form~$Q=(Q_v)_{v\in M_K}$ with~$Q_v\in X(K_v)$. Then by definition of
the ad\`elic (in this case, product) topology, a point~$Q\in
X(\Adele_K)$ is in~$\Closure\bigl(\Orbit_\f(P)\bigr)$ but not
in~$\Orbit_\f(P)$ if and only if 
there is an infinite set of positive integers~$\Ncal_{P,Q}\subset\NN$
such that for every~$v\in M_K$,
\begin{equation}
  \label{eqn:QvvlimfnP}
  Q_v = \vlim_{\substack{n\in\Ncal_{P,Q}\\n\to\infty\\}} \f^n(P).
\end{equation}
(We write~$\vlim$ to indicate that the limit is being taken in the
$v$-adic topology.)  N.B.~The set of integers~$\Ncal_{P,Q}$ depends
on~$P$ and~$Q$, but it must be independent of~$v\in M_K$.

\begin{remark}
We explain why it is necessary to assume some sort of condition on the
nontrivial $\f$-preperiodic subvarieties of~$V$.  Suppose for example
that~$V$ contains a nontrivial $\f$-preperiodic subvariety~$W$ and
suppose further that~$W(K)$ contains a point~$Q$ with
infinite~$\f$-orbit. We will construct a point~$P\in X(K)$ with the
property that
\begin{equation}
  \label{eqn:orbpneclos}
  \Orbit_\f(P) \cap V(K) 
  \ne \Closure\bigl(\Orbit_\f(P)\bigr) \cap V(\Adele_K).
\end{equation}
\par
Our assumptions mean that~$\f^n(W)=\f^m(W)$ for some $n>m$ and that
there is a non-preperiodic point~$Q\in W(K)$.  Replacing~$\f$,~$W$,
and~$Q$ by~$\f^{n-m}$,~$\f^m(W)$, and~$\f^m(Q)$, respectively, we have
\[
  \f(W)=W\qquad\text{and}\qquad Q\in W(K).
\]
The variety~$W$ is projective, so~$W(\Adele_K)$ is compact.
The infinite set $\Orbit_\f(Q)$ is contained in~$W(\Adele_K)$, so
its ad\`elic closure contains at least one accumulation point~$R$.
If~$R$ is not in~$\Orbit_\f(Q)$, then
\[
  R\notin\Orbit_\f(Q) \cap V(K) 
  \qquad\text{and}\qquad
  R\in \Closure\bigl(\Orbit_\f(Q)\bigr) \cap V(\Adele_K),
\]
so we are done.
\par
We are reduced to the case that some point~$\f^k(Q)$ in the orbit
of~$Q$ is an ad\`elic accumulation point of the orbit. We set
$P=\f^{k+1}(Q)$. Note that~$\f^k(Q)\notin\Orbit_\f(P)$, since we are
assuming that~$Q$ is not $\f$-preperiodic. On the other hand,
the accumulation points of~$\Orbit_\f(P)$ and~$\Orbit_\f(Q)$
in~$X(\Adele_K)$ are the same, since the two sets differ by only
finitely many elements. Hence
\[
  \f^k(Q)\notin\Orbit_\f(P) \cap V(K) 
  \qquad\text{and}\qquad
  \f^k(Q)\in \Closure\bigl(\Orbit_\f(P)\bigr) \cap V(\Adele_K),
\]
so in all cases we have constructed an orbit
satisfying~\eqref{eqn:orbpneclos}.
\end{remark}

\begin{remark}
\label{remark:vcontfnc}
We will make frequent use of the following elementary observation.
Let~$f$ be a rational function on~$X$ and let~$Z(f)$ be the support of
the polar divisor of~$f$. Then for any~$v\in M_K$, the function~$f$ is
continuous on~$X(K_v)\setminus Z(f)$ with respect to the~$v$-adic
topology. In particular, if the $v$-adic closure of the set
\text{$\bigl\{\f^n(P):n\in\Ncal_{P,Q}\bigr\}$} in~$X(K_v)$ is disjoint
from~$Z(f)$, then~\eqref{eqn:QvvlimfnP} implies that
\[
  f(Q_v) = \vlim_{n\in\Ncal_{P,Q}} f\bigl(\f^n(P)\bigr).
\]
\end{remark}

\section{Power maps and translated tori}
\label{section:transtori}

In this section we consider the case that~$X=\PP^2$ and~$\f$ is a
power map and we show that the rational points on a translated torus
are Brauer--Manin unobstructed.

\begin{theorem}
\label{thm:transtori}
let
\[
  \f:\PP^2\longrightarrow\PP^2,\qquad
  \f\bigl([X,Y,Z]\bigr) = [X^d,Y^d,Z^d]
\]
be the~$d^{\text{th}}$-power map for some~$d\ge2$.  Let~$k,\ell\ge1$,
let~$A,B\in K$, and let~$V\subset\PP^2$ be the curve
\begin{equation}
  \label{VAXkBYl}
  V : AX^k = BY^\ell.
\end{equation}
Further let~$P\in\PP^2(K)$ be a point whose~$\f$-orbit~$\Orbit_\f(P)$
is infinite.
Then one of the following two statements is true\textup:
\begin{parts}
\item[\upshape(i)]
$\Orbit_\f(P) \cap V(K) 
       = \Closure\bigl(\Orbit_\f(P)\bigr) \cap V(\Adele_K)$.
\vspace{5pt}
\item[\upshape(ii)]
The variety~$V$ is preperiodic for~$\f$, and
there exists an~$i\ge0$ such that \text{$\Orbit_\f(P)\cap\f^i(V)$} is an
infinite set.
\end{parts}
\end{theorem}

A key tool in the proof of Theorem~\ref{thm:transtori} is the
following classical result on the distribution of the multiplicative
orders of an element of~$K^*$ when reduced modulo primes.

\begin{theorem}
\label{thm:bang}
\textup{(Bang, Zsigmondy)}
Let~$K$ be a number field, let~$\l\in K^*$ be an element that
is not a root of unity, and let
\[
  S_\l = M_K^\infty \cup \{v\in M_K : |\l|_v\ne 1\}.
\]
For each~$v\notin S_\l$, let~$f_v(\l)$ denote the order of~$\l$
in~$\FF_v^*$, the multiplicative group of residue field at~$v$. Then
the set
\[
  \NN \setminus \{f_v(\l) : v\notin S_\l\}
\]
is finite, i.e., all but finitely many positive integers occur as the
order modulo~$\gp$ of~$\l$ for some prime~$\gp$ of~$K$.
\end{theorem}
\begin{proof}
This was originally proven by Bang~\cite{bang},
Zsigmondy~\cite{MR1546236}, and Birkhoff and Vandiver~\cite{MR1503541}
for~$K=\QQ$. It was extended to number fields by Postnikova and
Schinzel~\cite{MR0223330} and in strengthened form
by Schinzel~\cite{MR0344221}. 
\end{proof}

We use the Bang--Zsigmondy theorem to prove an ad\`elic property
of iterated power maps.

\begin{proposition}
\label{proposition:adelicpowermap}
Let $d\ge2$, 
let $\l,\xi\in K^*$, let~$S$ be a finite set of places of~$K$,
and let~$\Ncal$ be an infinite sequence of positive
integers. Suppose that for every~$v\notin S$ we have
\begin{equation}
\label{vlimoflx}
  \xi = \vlim_{\substack{n\in\Ncal\\n\to\infty\\}} \l^{d^n}.
\end{equation}
Then both~$\xi$ and~$\l$ are roots of unity, and there is
an integer~$r>0$ such that~$\xi=\l^{d^r}$.
\end{proposition}
\begin{proof}
If $\l$ is a root of unity, then the set $\{\l^{d^n}\}_{n\ge1}$ is
finite and consists entirely of roots of unity. Hence
the set~$\{\l^{d^n}\}_{n\in\Ncal}$ is a discrete subset of~$K_v$
(for any~$v$), so the existence of the limit~\eqref{vlimoflx} implies
that~$\xi$ is one of the roots of unity in this set.
\par
We assume henceforth that~$\l$ is not a root of unity and derive a
contradiction.  Without loss of generality, we may adjoin finitely
many places to~$S$, so we may assume that~$S$ contains all archimedean
places and that
\[
  |\l|_v = |\xi|_v = 1
  \quad\text{for all $v\notin S$.}
\]
\par
Theorem~\ref{thm:bang} tells us that that all but finitely many integers
appear as the order of~$\l$ modulo~$\gp_v$ for some place~$v\notin S$.
Hence after discarding finitely many elements from the set~$\Ncal$, we
find for every~$n\in\Ncal$ there is a distinct place~$v_n\notin S$
such that
\begin{equation}
  \label{eqn:fvndn}
  f_{v_n}(\l) = d^n.
\end{equation}
For notational convenience, we write~$\gp_n$ for the prime ideal 
associated to~$v_n$. Then~\eqref{eqn:fvndn} implies that
\[
  \l^{d^n} = \l^{f_{v_n}(\l)} \equiv 1 \pmod{\gp_n}.
\]
We also observe that if~$N\ge n$, then
\[
  \l^{d^N} = \left(\l^{d^n}\right)^{d^{N-n}}\equiv 1\pmod{\gp_n}.
\]
Hence if we choose any list of distinct
elements~$n_1,n_2,\ldots,n_t\in\Ncal$, then
\begin{equation}
  \label{equat2-1}
  \l^{d^N} \equiv 1 \pmod{\gp_{n_1}\gp_{n_2}\cdots\gp_{n_t}}
  \qquad\text{for all $N\ge\max\{n_1,n_2,\ldots,n_t\}$.}
\end{equation}
\par
We now use the assumption~\eqref{vlimoflx} that
\[
  \xi = \vlim_{n\in\Ncal} \l^{d^n}
  \qquad\text{for all $v\notin S$.}
\]
This says that~$\l^{d^n}$ is $v$-adically close to~$\xi$
for all large~$n\in\Ncal$, so in particular
\begin{align}
\label{equat2-2}
  \l^{d^n} \equiv \xi \pmod{\gp_i} 
  & \qquad \text{for all sufficiently large $n \in \Ncal$}\\[-1\jot]
  & \qquad \text{and all $1 \le i \le t$.} \notag
\end{align}
Combining \eqref{equat2-1} and \eqref{equat2-2} yields
\[
  \xi \equiv 1 \pmod{\gp_{n_1}\gp_{n_2}\cdots\gp_{n_t}}.
\]
But $t$ is arbitrary and the primes $\gp_{n_i}$ are distinct,
so $\xi=1$.
\par
Now \eqref{vlimoflx} becomes
\begin{equation}
  \label{2prime}
  \vlim_{n\in \Ncal} \l^{d^n} = 1 
  \qquad\text{for all $v\not\in S$,}
\end{equation}
and we want to derive a contradiction to the assumption that~$\l$ is
not a root of unity.  
\par
Applying Theorem~\ref{thm:bang} again, we find an infinite set of
positive integers~$\Mcal$ such that for 
\[
  \gcd(m,d)=1\quad\text{for every $m\in\Mcal$,}
\]
and such that for every~$m\in\Mcal$ there is a place~$w_m\notin S$
satisfying
\[
  f_{w_m}(\l) = m.
\]
Thus
\begin{equation}
  \label{eqn:lmfwm}
  \l^m = \l^{f_{w_m}(\l)} \equiv 1\pmod{\gp_m}.
\end{equation}
\par
On the other hand, the limit~\eqref{2prime} tells us that
for any~$m\in\Mcal$ there is a some~$n(m)$ such that
\begin{equation}
  \label{eqn:ldnm1}
  \l^{d^{n(m)}} \equiv 1 \pmod{\gp_m}.
\end{equation}
Combining~\eqref{eqn:lmfwm} and~\eqref{eqn:ldnm1} and using
the assumption that~$\gcd(m,d)=1$ for all~$m\in\Mcal$, we conclude
\[
  \l \equiv 1\pmod{\gp_m}
  \qquad\text{for all $m\in\Mcal$.}
\]
Since~$\Mcal$ is an infinite set and the~$\gp_m$ are distinct prime
ideals, this implies that~$\l=1$, contradicting the assumption
that~$\l$ is not a root of unity.
\end{proof}

We now have the tools needed to prove the main theorem of this
section.

\begin{proof}[Proof of Theorem~\ref{thm:transtori}]
We suppose that there is a point
\begin{equation}
  \label{eqn:QCOfPVAK}
  Q \in \Closure\bigl(\Orbit_\f(P)\bigr) \cap V(\Adele_K)
  \quad\text{with}\quad
  Q\notin\Orbit_\f(P)
\end{equation}
and will prove under this assumption that~$V$ is preperiodic
for~$\f$ and that~\text{$\Orbit_\f(P)\cap V$} is an infinite set.
As noted earlier, the point~$Q$ has the form~$Q=(Q_v)_{v\in
M_K}$, and the assumption
that~\text{$Q\in\Closure\bigl(\Orbit_\f(P)\bigr)\setminus\Orbit_\f(P)$}
means that there is an infinite set of positive
integers~$\Ncal_{P,Q}\subset\NN$ such that for every~$v\in M_K$,
\[
  Q_v = \vlim_{\substack{n\in\Ncal_{P,Q}\\n\to\infty\\}} \f^n(P).
\]
To ease notation, we will leave off the~$n\to\infty$ and the
dependence on~$P$ and~$Q$ and write simply~$\vlim_{n\in\Ncal}$ to mean
the $v$-adic limit as~$n\to\infty$ with~$n\in\Ncal_{P,Q}$.
\par
Let~$P=[\a,\b,\g]$.  We consider first the case that~$\a\b\g\ne0$, so
we may dehomogenize by  setting~$\g=1$.
Let~$S$ be the set
\[
   S =  M_K^\infty \cup \{v\in M_K : |\a|_v\ne1\}
         \cup \{v\in M_K : |\b|_v\ne1\},
\]
so in particular,~$\a$ and~$\b$ are both~$S$-units.  (In the notation of
Theorem~\ref{thm:bang}, we have~\text{$S=S_\a\cup S_\b$}.)
\par
The fact that
\[
  Q_v = \vlim_{n\in\Ncal} \f^n(P) 
    = \vlim_{n\in\Ncal} \bigl[\a^{d^n},\b^{d^n},1\bigr]
\]
implies that for all~$v\notin S$, the point~$Q_v$ has the form
\[
  Q_v = [x_v,y_v,1]\quad\text{with}\quad
  |x_v|_v = |y_v|_v = 1,
\]
and further the sequences~$\{\a^{d^n}\}_{n\in\Ncal}$
and~$\{\b^{d^n}\}_{n\in\Ncal}$ converge~$v$-adically in~$K_v$ with
\[
  x_v = \vlim_{n\in\Ncal} \a^{d^n} \quad\text{and}\quad y_v 
    = \vlim_{n\in\Ncal} \b^{d^n}.
\]
\par
Let~$f(x,y)=Ax^k-By^\ell$. We view~$f$ as a rational function
on~$\PP^2$ and observe that~$f(x,y)=0$ is an affine equation for~$V$.
Since the quantities~$\a$,~$\b$,~$x_v$, and~$y_v$ are all~$v$-adic
units (remember that are assuming that~$v\notin S$), it follows
that the $v$-adic closure of the set
\[
  \bigl\{\f^n(P) : n\in\Ncal\bigr\}
\]
is disjoint from the polar divisor of~$f$. Also by assumption we
have~$Q_v\in V(K_v)$, so~$f(Q_v)=0$. Hence applying
Remark~\ref{remark:vcontfnc}, we find that
\begin{align*}
  0 &= \vlim_{n\in\Ncal} f\bigl(\f^n(P)\bigr) \\
    &=  A \cdot \bigl(\vlim_{n\in\Ncal} \a^{d^n}\bigr)^k 
      - B\cdot \bigl(\vlim_{n\in\Ncal} \b^{d^n}\bigr)^\ell.
\end{align*}
It follows that~$A$ and~$B$ are nonzero, since~$\a$ and~$\b$
are~$v$-units, so a little bit of algebra yields
\begin{equation}
  \label{vlimakbldnBA}
  \vlim_{n\in\Ncal} \left(\frac{\a^k}{\b^\ell}\right)^{d^n} = \frac{B}{A}.
\end{equation}
\par
Now Proposition~\ref{proposition:adelicpowermap} tells us that~$B/A$
and~$\a^k/\b^\ell$ are roots of unity and that there is an integer~$r\ge1$
such that
\begin{equation}
  \label{eqn:akbldrBA}
  \left(\frac{\a^k}{\b^\ell}\right)^{d^r} = \frac{B}{A}.
\end{equation}
\par
The fact that~$B/A$ is a root of unity implies that~$V$ is
preperiodic for~$\f$.  To see this, we write~$V_{[A,B]}$ in
order to indicate the dependence of~$V$ on the
parameter~$[A,B]\in\PP^1$. It is clear from the definition of~$\f$
and~$V$ that $\f\bigl(V_{[A,B]}\bigr)= V_{[A^d,B^d]}$, and hence
\[
  \f^n\bigl(V_{[A,B]}\bigr)= V_{[A^{d^n},B^{d^n}]}.
\]
It thus suffices to find~$n>m$ satisfying
$[A^{d^n},B^{d^n}]=[A^{d^m},B^{d^m}]$, which can be done since~$B/A$ is
a root of unity. Hence~$V$ is preperiodic for~$\f$.
\par
Let~$i\ge0$ be an integer so that~$\f^i(V)$ is periodic for~$\f$, say
with period~$q$.  The formula~\eqref{eqn:akbldrBA} says
that~$\f^r(P)\in V$, so we find that $\f^{i+r+qj} (P)\in \f^i(V)$ for
all~$j\ge0$.  This proves that~$\Orbit_\f(P)\cap\f^i(V)$ is an
infinite set, which
completes the proof of Theorem~\ref{thm:transtori} under the
assumption that~$\a\b\g\ne0$.
\par
It remains to deal with the case~$\a\b\g=0$, where we recall that~$P$
is the point \text{$P=[\a,\b,\g]\in\PP^2$}. Suppose first
that~$\g=0$. The assumption that~$\Orbit_\f(P)$ is infinite implies
that~$\a\b\ne0$, so we can dehomogenize and write \text{$P=[\a,1,0]$}
with~$\a\ne0$. Then just as in the case~$\g=1$, we find that
\[
  Q_v = \vlim_{n\in\Ncal}\bigl[  \a^{d^n}, 1, 0 \bigr]
  \qquad\text{for all $v\notin S$,}
\]
and the fact that~$Q_v\in V$ tells us that
\[
  \vlim_{n\in\Ncal} \a^{d^n} = B/A.
\]
Applying Proposition~\ref{proposition:adelicpowermap} again, we
conclude that~$\a$ and~$B/A$ are roots of unity. But this implies the point
  $P = [\a, 1, 0]$ is preperiodic for $\f$, contradicting to our
  assumption that $P$ has infinite orbit.
\par
Next suppose that~$\g\ne0$ and~$\b=0$. We dehomogenize~$P$ in the
form \text{$P=[\a,0,1]$}, and then
\[
  Q_v = \vlim_{n\in\Ncal}\bigl[ \a^{d^n}, 0, 1 \bigr].
\]
Taking any~$v$ with~$|\a|_v=1$, the fact that~$Q_v=[x_v,0,1]\in V$
implies that~$Ax_v^k=B\cdot0^\ell=0$, so~$A=0$.  (Note that~$x_v\ne0$
since $|x_v|_v=1$.)  Thus~$V$ is the line~$Y=0$, i.e., it is given by
the equation~$BY^\ell=0$, so it is fixed by~$\f$ and the entire 
orbit~$\Orbit_\f(P)$ lies on~$V$.
\par
Finally, if~$\a=0$, the same argument shows that~$B=0$, so~$V$ is the
line~$X=0$, hence is fixed by~$\f$ and $\Orbit_\f(P)\subset V$.
\end{proof}

\section{Power maps and linear varieties}
\label{section:powerline}
In this section we again take~$\f$ to be the a power map
\[
  \f:\PP^2\longrightarrow\PP^2,\qquad
  \f\bigl([X,Y,Z]\bigr) = [X^d,Y^d,Z^d]
\]
and consider a line~$V\subset\PP^2$ given by an equation
\begin{equation}
  \label{VAXBYCZ}
  V : AX + BY + CZ = 0.
\end{equation}
Let
\begin{equation}
  \label{Pabg}
  P = [\a,\b,\g]
\end{equation}
be the given point with infinite~$\f$-orbit.  
\par
Note that if~$ABC=0$, then Theorem~\ref{thm:transtori} says
that~$V(K)$ is Brauer--Manin unobstructed.  We now prove analogous
results for~$ABC\ne0$ when~$P$ has various special forms.

\begin{theorem}
\label{thm:Pprepercoord}
Let~$V$ and~$P$ be as given in~\eqref{VAXBYCZ} and~\eqref{Pabg} and
assume that~$ABC\ne0$. Further assume that one of the following is
true\textup:
\begin{parts}
\Part{(a)}
One of the ratios~$\a/\b$,~$\b/\g$, or~$\a/\g$ is a root of unity.
\Part{(b)}
One of the coordinates~$\a$,~$\b$, or~$\g$ is zero.
\end{parts}
Then
\[
  \Orbit_\f(P) \cap V(K) = \Closure\bigl(\Orbit_\f(P)\bigr) \cap V(\Adele_K).
\]
\textup(It is easy to check that the assumption that~$ABC\ne0$ implies
that~$V$ cannot be preperiodic for~$\f$.\textup)
\end{theorem}
\begin{proof}
(a) It suffices to consider the case that~$\b/\g$ is a root of unity.
Dividing the equation of~$V$ by~$-CZ$ and the coordinates of~$P$ by~$-\g$,
without loss of generality we can use affine coordinates of the form
\[
  V:Ax+By=1\qquad\text{and}\qquad P=(\a,\b).
\]
Our assumptions are that~$AB\ne0$ and~$\b$ is a root of unity.
Further, since~$P$ has infinite~$\f$-orbit, it follows that~$\a\ne0$
and~$\a$ is not a root of unity.
\par
We assume that there is a point
\[
  Q \in \Closure\bigl(\Orbit_\f(P)\bigr) \cap V(\Adele_K)
  \quad\text{with}\quad Q\notin\Orbit_\f(P)
\]
and derive a contradiction. (We will show that this forces~$\a$ to be
a root of unity.)  As in the proof of Theorem~\ref{thm:transtori}, we
let~$\Ncal=\Ncal_{P,Q}$ be an infinite set of integers so that
\[
  Q_v = (x_v,y_v)  = \vlim_{n\in\Ncal}\f^n(P)  = \vlim_{n\in\Ncal} 
    \bigl(\a^{d^n},\b^{d^n}\bigr)
  \quad\text{for all $v\in M_K \setminus S_\a$.}
\]
\par
The assumption that~$\b$ is a root of unity implies that~$\b^{d^n}$ takes
on only finitely many distinct values, so replacing~$\Ncal$ with a subsequence,
we may assume that~$\b^{d^n}=\b_0$ is constant for all~$n\in\Ncal$. Thus
\[
  x_v = \vlim_{n\in\Ncal}\a^{d^n}
  \quad\text{and}\quad
  y_v = \b_0
  \quad\text{for all $v\in M_K \setminus  S_\a$.} 
\]
Note that~$\b_0$ is a root of unity.
\par
The fact that~$Q_v\in V$ tells us that for all~$v$ with~$|\a|_v=1$.
\[
  \vlim_{n\in\Ncal}\a^{d^n} 
  = x_v = \frac{1-B\b_0}{A}
  \quad\text{for all $v\in M_K \setminus  S_\a$.}
\]
It follows from Proposition~\ref{proposition:adelicpowermap}
that~$\a$ is a root of unity, contradicting our assumption that~$P$
has infinite~$\f$-orbit.
\par\noindent(b)\enspace
By symmetry, we may assume that~$\g=0$.  Then the fact that~$P$ is not
preperiodic implies that~$\a\b\ne0$ and~$\a/\b$ is not a root of unity.
Dehomogenizing~$P$ with respect to the~$Y$-coordinate, we can
write~$P$ in the form~$P=[\a,1,0]$ with~$\a$ not a root of unity.
Suppose now that 
\[
  \Orbit_\f(P) \cap V(K) 
  \ne \Closure\bigl(\Orbit_\f(P)\bigr) \cap V(\Adele_K).
\]
Then as in the proofs of Theorem~\ref{thm:transtori} and part~(a) of
this proposition, there is a finite set of places~$S$ so that
\[
  A \vlim_{n\in\Ncal} \a^{d^n} + B\cdot 1 = C\cdot 0
  \qquad\text{for all $v\notin S_\a$.}
\]
We have~$AB\ne0$ by assumption, so
Proposition~\ref{proposition:adelicpowermap} tells us
that~$\a$ is a root of unity, contradicting our assumption
that~$P$ has infinite~$\f$-orbit.
\end{proof}

\section{Abelian Varieties and Translated Abelian Subvarieties}
\label{section:transabvar}

In this section we prove an analog of Theorem~\ref{thm:transtori} for
abelian varieties. The key tool will be the following elliptic analog
of the Bang--Zsigmondy theorem (Theorem~\ref{thm:bang}).

\begin{theorem}
\label{thm:ellipticbz}
Let~$E/K$ be an elliptic curve defined over a number field~$K$,
let~$P\in E(K)$ be a nontorsion point, and let~$S\subset M_K$ be a
finite set of places including~$M_K^\infty$ and all place of bad
reduction for~$E$.  For each place~$v\notin S$, let~$f_v(P)$ be the
order of~$P\pmod{\gp_v}$ in~$E(\FF_{\gp_v})$. Then the set
\[
  \NN\setminus\bigl\{f_v(P) : v\notin S\bigr\}
\]
is finite.
\end{theorem}
\begin{proof}
See~\cite{MR961918} for the case~$K=\QQ$ and~\cite{MR1683630} for
general number fields.  We note that the assertion of
Theorem~\ref{thm:ellipticbz} is not an elementary fact, since its
proof requires a strong form of Siegel's
theorem~\cite[IX.3.1]{MR1329092} on integral points on elliptic
curves.
\end{proof}

\begin{theorem}
\label{thm:transav}
Let~$K/\QQ$ be a number field, let~$A/K$ be an abelian variety, and
let~$B/K$ be an abelian subvariety of~$A$ of codimension~$1$.  We fix
a point~$T\in A(K)$ and let $V = B + T$ be the translation of $B$
by~$T$.
\par
Let $d \ge 2$ be an integer and consider the multiplication-by-$d$
map
\[
  [d]:A\to A. 
\]
Let $P \in A(K)$ be a nontorsion point.  Then one of the following two
statements is true\textup:
\begin{parts}
\item[\upshape(i)]
$\Orbit_d(P) \cap V(K) 
       = \Closure\bigl(\Orbit_d(P)\bigr) \cap V(\Adele_K)$.
\vspace{5pt}
\item[\upshape(ii)]
The variety~$V$ is~$[d]$-preperiodic.
\end{parts}
Further, in case~\textup{(ii)}, the point~$T$ has finite order in the
quotient variety~$A/B$.
\end{theorem}

\begin{proof}
Suppose that~(i) is false, and let
\[
  (Q_v)_{v\in M_K} 
  \in \Closure\bigl(\Orbit_\f(P\bigr)) \cap  V(\Adele_K)
  \quad\text{with}\quad
  (Q_v)_{v\in M_K}  \notin \Orbit_\f(P)\cap V(K).
\]
As in the proof of Theorem~\ref{thm:transtori}, this means that there
is an infinite set of positive integers $\Ncal \subset \NN$ such that
for every $v\in M_K$,
\[
  Q_v 
  =  \vlim_{\substack{n\in\Ncal \\n\to\infty\\}} [d^n](P) 
  \in V(K_v).
\]
\par
We now pass to the quotient abelian variety $E = A/B$.  Since $B$ has
codimension~$1$ in $A$, it follows that $E$ is an abelian variety of
dimension~$1$, i.e., $E$ is an elliptic curve.  For convenience, we
use bars to denote the image of points of~$A$ in~$E$.
Note that~$Q_v\in V = B+T$, so on the quotient variety we see that
\[
  \bar{Q}_v = \Tbar \in E
\]
is independent of~$v$. Hence for every $v\in M_K$ we have
\begin{equation}
  \label{eqn:vlimnndnPbar}
  \vlim_{\substack{n\in\Ncal \\ n\to\infty\\}} [d^n](\Pbar)
  = \Tbar
  \quad\text{in the $v$-adic topology,}
\end{equation}
so in particular, for every~$v\in M_K$ there is an~$N_v$ so that
\begin{equation}
  \label{eqn:dnPbTbgpv}
  [d^n](\Pbar) \equiv \Tbar \pmod{\gp_v}
  \qquad\text{for all $n\in\Ncal$ with $n\ge N_v$.}
\end{equation}
\par
We consider two cases. First, suppose that~$\Pbar\in E(K)$ is a
nontorsion point. Then we can apply the elliptic Zsgmondy theorem
(Theorem~\ref{thm:ellipticbz}) to~$\Pbar\in E(K)$.
This tells us that for all but finitely many~$n\in\Ncal$ 
there is a place~$v_n\in M_K$ such that
\begin{equation}
  \label{eqn:fvnPbardn}
  f_{v_n}(\Pbar) = d^n,
\end{equation}
i.e., the point~$\Pbar\bmod\gp_{v_n}$ has order~$d^n$ 
in~$E(\FF_{\gp_{v_n}})$. Hence if~$m\ge n$, then
\begin{equation}
  \label{eqn:dmPOpvn}
  [d^m](\Pbar)
  = [d^{m-n}][d^n](\Pbar)
  = [d^{m-n}][f_{v_n}](\Pbar) 
  \equiv \bar{O} \pmod{\gp_{v_n}}.
\end{equation}
\par
Now choose distinct integers~$n_1,n_2,\dots,n_t\in\Ncal$ so
that~\eqref{eqn:fvnPbardn} is true and, to ease notation,
let~~$v_1,v_2,\dots,v_t$ be the associated valuations.
Then for any integer~$m\in\Ncal$ satisfying
\[
  m \ge \max\{N_{v_1},N_{v_2},\dots,N_{v_t},n_1,n_2,\ldots,n_t\}
\]
we have for every~$1\le i\le t$,
\begin{align*}
  [d^m](\Pbar) &\equiv \Tbar \pmod{\gp_{v_i}}
    &&\text{from \eqref{eqn:dnPbTbgpv},} \\
  [d^m](\Pbar) &\equiv \bar{O} \pmod{\gp_{v_i}}
    &&\text{from \eqref{eqn:dmPOpvn}.}
\end{align*}
Hence
\[
  \Tbar \equiv \bar{O}\pmod{\gp_{v_i}}
  \qquad\text{for all $1\le i\le t$.}
\]
Since~$t$ is arbitrary, we conclude that~$\Tbar=\bar{O}$.
Thus~$T\in B$, which implies that~$V=B+T=B$, since~$B$ is an
abelian subvariety of~$A$. In particular,~$V$ is preperiodic for~$[d]$,
since in fact~$[d](V)=V$.
\par
Next we suppose that~$\Pbar$ is a torsion point of $E(K)$.
Then the set of points~$\bigl\{[d^n](\Pbar) : n\in\Ncal\bigr\}$ is finite,
so the existence of the limit~\eqref{eqn:vlimnndnPbar} implies
that~$\Tbar$ is one of the torsion points appearing in this set.
Hence
\[
  [d^n](V) = [d^n](B+T) = B+ [d^n](T)
\]
takes on only finitely many values, so~$V$ is preperiodic for~$[d]$.
\end{proof}

\section{Abelian Varieties and General Subvarieties}
\label{section:av}

In this section we prove a local-global result for dynamical systems
on abelian varieties that is related to the Brauer--Manin
obstruction. Roughly speaking, we assume that~$V(K)$ is Brauer--Manin
unobstructed in~$A(K)$ (as described in Theorem~\ref{thm:abv}~(ii)
below) and we prove that orbits of multiplication maps are
Brauer--Manin unobstructed. We set the following notation

\begin{notation}
\item[$K/\QQ$]
a number field.
\item[$A/K$]
an abelian variety.
\item[$V/K$]
a subvariety of $A/K$.
\item[$A(K_v)^0$]
identity component of~$A(K_v)$ for $v\in M_K^\infty$,
$0$ for~$v\in M_K^0$.
\item[$\Adot$]
${}=\prod_{v\in M_K} A(K_v)/A(K_v)^0$.
\item[$\Vdot$]
  the image of $V(\Adele_K)$ under the projection $A(\Adele_K) \to
  \Adot .$
\end{notation}


\begin{theorem}
\label{thm:abv}
We make the following assumptions\textup:
\begin{itemize}
\item[\textup(i)]
$V$ does not contain a translate of a positive-dimensional abelian
subvariety of~$A$.
\item[\textup(ii)]
$V(K)=\Vdot\cap\Closure\bigl(A(K)\bigr)$, where the closure and
the equality take place in~$\Adot$.
\end{itemize}
Then for all integers~$d\ge2$ and all nontorsion points~$P\in A(K)$ we
have
\[
  V(K)\cap \Orbit_d(P) = \Vdot\cap\Closure\bigl(\Orbit_d(P)\bigr).
\]
\textup(Here $\Orbit_d$ denotes the orbit under 
the multiplication-by-$d$ map on~$A$.\textup) 
\end{theorem}

The proof of Theorem~\ref{thm:abv} uses the following results
of Serre and Stoll. 


\begin{theorem}
\label{thm:serrestoll}
Let
\[
  \Ahat(K)=A(K)\otimes\hat{\ZZ}
\]
be the profinite completion of $A(K)$, and let~$S$ be a set of
places of~$K$ of density~$1$.
\begin{parts}
\Part{(a)}
The map
\[
  \Ahat(K) \longrightarrow \prod_{v\in S} A(K_v)/A(K_v)^0
\]
is injective.
\Part{(b)}
The natural map~$\Ahat(K)\to\Adot$ induces an isomorphism
between~$\Ahat$ and~$\Closure\bigl(A(K)\bigr)$, the topological
closure of~$A(K)$ in~$\Adot$.
\Part{(c)}
Assume further that~$S\subset M_K^0$ and that~$A$ has good reduction
at every place in~$S$. Then the composition of the natural maps
\[
  \Ahat(K) \longrightarrow \prod_{v\in S} A(K_v)
  \longrightarrow \prod_{v\in S} A(\FF_v)
\]
is injective.
\end{parts}
\end{theorem}
\begin{proof}
(a) and (b) are part of~\cite[Theorem~3]{MR0289513}, and we
refer to~\cite[Proposition~3.7]{arxiv0606465v3} for the proof.
The statement in~(c) is part of~\cite[Theorem~3.10]{arxiv0606465v3}.
\end{proof}

\begin{corollary}
\label{corXKH}
Let~$H$ be a subgroup of~$A(K)$ and let~$\Closure(H)$ be the closure
of~$H$ inside~$\Adot$. Then under the assumption
\begin{equation}
  \label{eqn:VKVdCAK}
  V(K)=\Vdot\cap\Closure\bigl(A(K)\bigr)
\end{equation}
from Theorem~$\ref{thm:abv}$, we have
\[
  V(K)\cap H = \Vdot \cap \Closure(H).
\]
\end{corollary}
\begin{proof}
Using the assumption~\eqref{eqn:VKVdCAK}, we have
\begin{equation}
  \label{eqn:VKHVK}
  V(K)\cap H 
  \subset \Vdot\cap\Closure(H)
  \subset \Vdot \cap \Closure\bigl(A(K)\bigr)
  = V(K),
\end{equation}
Now let \text{$Q\in\Vdot\cap\Closure(H)\subset V(K)$} and suppose
that~$Q\notin H$.  Then by definition there is a sequence of points
\[
  P_1,P_2,P_3,\dots\in H
  \quad\text{such that}\quad
  \lim_{n\to\infty} P_n = Q\quad\text{in the ad\`elic topology.}
\]
In particular, if~$A$ has good reduction at~$v\in M_K^0$, then
\[
  Q \equiv P_n \pmod{\gp_v}
  \quad\text{for all sufficiently large $n$.}
\]
But every~$P_n$ is in~$H$, so we conclude that for all but finitely
many~$v\in M_K$ we have
\[
   Q \bmod\gp_v \in H\bmod\gp_v.
\]
\par
The group~$A(K)$ has no non-trivial divisible elements, so the 
natural map \text{$A(K)\to\Ahat(K)$} is an injection. 
It follows from Theorem~\ref{thm:serrestoll}(c) that \text{$Q\in H$}.
This contradiction proves the inclusion
\[
  \Vdot\cap\Closure(H)\subset H,
\]
which completes the proof of Corollary~\ref{corXKH}.
\end{proof}

We now have the tools needed to prove the main result of this section.

\begin{proof}[Proof of Theorem~\ref{thm:abv}]
Let~$H=\ZZ P$ be the (free) subgroup of~$A(K)$ generated by~$P$,
so in particular \text{$\Orbit_d(P)\subset H$}.
Then
\begin{equation}
  \label{eqn:VCdPVKH}
  \Vdot \cap \Closure\bigl(\Orbit_d(P)\bigr)
  \subset \Vdot\cap\Closure(H) = V(K)\cap H.
\end{equation}
where the equality is from Corollary~\ref{corXKH}.  Now let
\[
  Q\in \Vdot \cap \Closure\bigl(\Orbit_d(P)\bigr).
\]
It follows from~\eqref{eqn:VCdPVKH} that~$Q\in V(K)$ is a~$K$-rational
point of~$V$ and further that~$Q\in H$, so~$Q=mP$ for some integer~$m$.
\par
Suppose now that~\text{$Q\notin\Orbit_d(P)$}. Then we can find an
infinite sequence of positive integers~$\{r_1,r_2,r_3,\dots\}$ so that
\[
  \lim_{i\to\infty} d^{r_i}P = Q = mP
  \quad\text{in the ad\`elic topology.}
\]
Theorem~\ref{thm:serrestoll}(a,b) tells us that
\[
  \Closure(H) = H \otimes\hat\ZZ,
\]
which implies that $d^{r_i} \to m$ in~$\hat\ZZ$ as~$i\to\infty$.
Hence for every prime~$p$ we have
\[
  m = \operatornamewithlimits{\text{$p$}-lim}_{i\to\infty} d^{r_i}.
\]
Taking~$p$ to be any prime dividing~$d$ (this is where we use the
assumption that~$d\ge2$), we find that~$m=0$. On the other hand,
if~$p$ is a prime not dividing~$d$, then~$|d|_p=1$, so~$|m|_p=1$.
This contradiction shows that
\[
  \Vdot \cap \Closure\bigl(\Orbit_d(P)\bigr)
  \subset \Orbit_d(P),
\]
which completes the proof of Theorem~\ref{thm:abv}.
\end{proof}

\begin{remark}
More generally, let $\f:A\to A$ be a $K$-endomorphism of~$A$ with the
property that the subring $\ZZ[\f]\subset\End(A)$ is an integral
domain. (If~$A$ is geometrically simple, this will be true for any
endomorphism~$\f$ of infinite order.) Then we can repeat the proof of
Theorem~\ref{thm:abv}, \emph{mutatis mutandis}, using the
subgroup~$H=\ZZ[\f]P$ and the fact that the fraction field
of~$\ZZ[\f]$ is a number field. Note that if~$A$ is not simple, then
it is possible for~$\ZZ[\f]$ to have zero-divisors.
\end{remark}





\end{document}